\documentclass[preprint,review,12pt]{elsarticle}
\usepackage{amssymb,amsmath,amsthm,enumerate,amscd,graphicx,color}
 \usepackage[colorlinks=true, pdfstartview=FitV, linkcolor=blue, citecolor=blue, urlcolor=blue]{hyperref}
\usepackage{pdfsync}
\numberwithin{equation}{section}
\newtheorem{thm}{Theorem}[subsection]

\newtheorem{cor}[thm]{Corollary}
\newtheorem{prop}[thm]{Proposition}

\newcommand\Aboxed[1]{  \@Aboxed#1\ENDDNE}
\def\@Aboxed#1&#2\ENDDNE{%
   &
   \settowidth\@tempdima{$\displaystyle#1{}$}
   \setlength\@tempdima{\@tempdima+\fboxsep+\fboxrule}
   \kern-\@tempdima
   \boxed{#1#2}
}
\newcommand{\thmref}[1]{Theorem~\ref{#1}}

\newcommand{\cororef}[1]{Corollary~\ref{#1}}

\newcommand{\eqnref}[1]{~(\ref{#1})}

\begin{document}

\title{DJKM algebras and associated ultraspherical polynomials}
\author{Ben Cox}
\address{Department of Mathematics \\
College of Charleston \\
66 George St.  \\
Charleston, SC 29424, USA}\ead{coxbl@cofc.edu}

\author{Vyacheslav Futorny}
\address{Department of Mathematics\\
 University of S\~ao Paulo\\
 S\~ao Paulo, Brazil}
 \ead{futorny@ime.usp.br}
 
\author{Juan A.Tirao}
\address{CIEM-FAMAF \\Universidad Nacional de Cordoba\\
5000 Cordoba \\Argentina}\ead{tirao@famaf.unc.edu.ar}
 \begin{abstract} We describe families of  polynomials arising in the study of the universal central extensions of Lie algebras introduced by Date, Jimbo, Kashiwara, and Miwa \cite{MR701334} in their work on the Landau-Lifshitz equations.  We show these two families of polynomials satisfy certain fourth order linear differential equations by direct computation and one of the families is a particular collection of associated ultraspherical polynomials.
 \end{abstract}


\maketitle

\section{Introduction}
Date, Jimbo, Kashiwara and Miwa \cite{MR701334} studied
integrable systems arising from Landau-Lifshitz differential equation. The hierarchy of this equation is written in terms of free fermions on an elliptic curve. The authors introduced a certain
infinite-dimensional Lie algebra which is a one dimensional central extension of $\mathfrak g\otimes \mathbb C[t,t^{-1},u|u^2=(t^2-b^2)(t^2-c^2)]$ where $b\neq \pm c$ are complex
 constants and $\mathfrak g$ is a simple finite dimensional Lie algebra. This Lie algebra which we call the DJKM algebra, acts on the solutions of the Landau-Lifshitz equation as infinitesimal B\"acklund transformations.

The Lie algebra above is an example of a Krichever-Novikov  algebra (see (\cite{MR902293}, \cite{MR925072}, \cite{MR998426}).  A fair amount of interesting and
fundamental work has be done by Krichever, Novikov, Schlichenmaier, and Sheinman on the representation theory of these algebras.
 In particular Wess-Zumino-Witten-Novikov theory and analogues of the Knizhnik-Zamolodchikov equations are developed for  these algebras
(see the survey article \cite{MR2152962}, and for example \cite{MR1706819}, \cite{MR1706819},\cite{MR2072650},\cite{MR2058804},\cite{MR1989644}, and \cite{MR1666274}).

In \cite{MR2813377} the authors gave commutations relations of
the universal central extension of the DJKM Lie algebra in terms of a basis of the algebra and certain polynomials.  More precisely in order to pin down this central extension, we needed to describe four \color{black} families of polynomials that appeared as coefficients in the commutator formulae.   In this previous work we gave recursion relations for these polynomials and then found generating functions for them.
Two of these families of polynomials are given in terms of
elliptic integrals and the other two families are slight variations of ultraspherical polynomials.
The main purpose of this note is to describe fourth order linear differential equations satisfied by the these two elliptic families of polynomials (see \eqnref{4thorder1} and \eqnref{4thorder2}), and to explain why these polynomials are orthogonal and nonclassical (see \cororef{cor-orthog-non-clas}, \thmref{thm-orth}, and \thmref{nonclassical}).   In fact one of the families are particular examples of associated ultraspherical polynomials (see \cite{MR663314}).   The associated ultraspherical polynomials in turn are, up to factors of ascending factorials, particular associated Jacobi polynomials.   The associated Jacobi polynomials are known to satisfy certain fourth order linear differential equations (see \cite{MR2191786} and formula (48) in \cite{MR915027}).  The proof given in the previous cited paper is obtained using MACSYMA and ideas coming from the classic text \cite{MR1349110}.    We show how we obtained a proof in the ultraspherical case that is done by hand and omit the calculation done for the remaining family of polynomials as it is similar. The referee was instrumental in identifying one of the families of nonclassical polynomials describing the universal central extension and the references about associated ultraspherical and Jacobi polynomials.  We would like to thank the referee for directing us to the relevant articles and books. 

For other examples of families of orthogonal polynomials that satisfy fourth order linear differential equations see the work of Samuel Shore, Allan Krall, and H. L. Krall (\cite{MR668944}, \cite{MR606336}, \cite{MR0002679}).  The authors would like to thank Lance Littlejohn for these references and helpful correspondence.    Still for even more examples of families of orthogonal polynomials that satisfy fourth order linear  differential equations appearing in the study of supersingular $j$-invariants see the work of Kaneko and Zagier (\cite{MR1486833}).

We plan to use these families of polynomials to describe free field realizations of the DKJM algebra in the setting of conformal field theory.  In the case of affine Kac-Moody algebras, initial motivation for the use of Wakimoto's realization (or free field realization) was to prove a conjecture of V. Kac and D. Kazhdan on the character of certain irreducible representations of affine Kac-Moody algebras at the critical level (see \cite{W} and \cite{MR2146349}). Another motivation for constructing free field realizations is that they are used to provide integral solutions to the KZ-equations (see for example \cite{MR1077959} and \cite{MR1629472} and their references).  A third is that they are used to help in determining the center of a certain completion of the enveloping algebra of an affine Lie algebra at the critical level which is an important ingredient in the geometric Langland's correspondence \cite{MR2332156}.  Yet a fourth is that free field realizations of an affine Lie algebra appear naturally in the context of the generalized AKNS hierarchies \cite{MR1729358}.

\section{DJKM algebras}

 Let $R$ be a commutative algebra defined over $\mathbb C$.
Consider the left $R$-module  with  action $f( g\otimes h ) = f g\otimes h$ for $f,g,h\in R$ and let $K$  be the submodule generated by the elements $1\otimes fg  -f \otimes g -g\otimes f$.
Then $\Omega_R^1=F/K$ is the module of K\"ahler differentials.  The element $f\otimes g+K$ is traditionally denoted by $fdg$.  The canonical map $d:R\to \Omega_R^1$ by $df  = 1\otimes f  + K$.
The {\it exact differentials} are the elements of the subspace $dR$.  The coset  of $fdg$  modulo $dR$ is denoted by $\overline{fdg}$.
As C. Kassel showed the universal central extension of the current algebra $\mathfrak g\otimes R$ where $\mathfrak g$ is a simple finite dimensional
Lie algebra defined over $\mathbb C$, is the vector space $\hat{\mathfrak g}=(\mathfrak g\otimes R)\oplus \Omega_R^1/dR$ with Lie bracket given by
$$
[x\otimes f,Y\otimes g]=[xy]\otimes fg+(x,y)\overline{fdg},  [x\otimes f,\omega]=0,  [\omega,\omega']=0,
$$
  where $x,y\in\mathfrak g$, and $\omega,\omega'\in \Omega_R^1/dR$ and $(x,y)$  denotes the Killing  form  on $\mathfrak g$.

Consider the polynomial
$$
p(t)=t^n+a_{n-1}t^{n-1}+\cdots+a_0
$$
where $a_i\in\mathbb C$ and $a_n=1$.
Fundamental to the description of the universal central extension for $R=\mathbb C[t,t^{-1},u|u^2=p(t)]$ are the following two results:
\begin{thm}[\cite{MR1303073},Theorem 3.4]  Let $R$ be as above.  The set
$$
\{\overline{t^{-1}\,dt},\overline{t^{-1}u\,dt},\dots, \overline{t^{-n}u\,dt}\}
$$
 forms a basis of $\Omega_R^1/dR$ (omitting $\overline{t^{-n}u\ dt}$ if $a_0=0$).
\end{thm}

\begin{prop}[\cite{MR2813377}, Lemma 2.0.2.]  If $u^m=p(t)$ and $R=\mathbb C[t,t^{-1},u|u^m=p(t)]$, then in $\Omega_R^1/dR$, one has
\begin{equation}\label{recursionreln}
((m+1)n+im)t^{n+i-1}u\,dt \equiv - \sum_{j=0}^{n-1}((m+1)j+mi)a_jt^{i+j-1}u\,dt\mod dR
\end{equation}
\end{prop}

In the Date-Jimbo-Miwa-Kashiwara  setting one takes $m=2$ and $p(t)=(t^2-a^2)(t^2-b^2)=t^4-(a^2+b^2)t^2+(ab)^2$ with $a\neq \pm b$ and neither $a$ nor $b$ is zero.
We fix from here onward $R=\mathbb C[t,t^{-1},u\,|\,u^2= (t^2-a^2)(t^2-b^2)]$.  As in this case $a_0=(ab)^2$, $a_1=0$, $a_2=-(a^2+b^2)$, $a_3=0$ and $a_4=1$, then
 letting $k=i+3$ the recursion relation in \eqnref{recursionreln} looks like

\begin{align*}
(6+2k)\overline{t^{k}u\,dt}
&=-2(k-3)(ab)^2\overline{t^{k-4}u\,dt} +2k(a^2+b^2)\overline{t^{k-2}u\,dt}.
\end{align*}
After a change of variables, $u\mapsto u/ab$, $t\mapsto t/\sqrt{ab}$, we may assume that $a^2b^2=1$.   Then the recursion relation looks like
\begin{equation}\label{recursionreln1}
(6+2k)\overline{t^{k}u\,dt}
=-2(k-3)\overline{t^{k-4}u\,dt} +4kc\overline{t^{k-2}u\,dt},
\end{equation}
after setting $c=(a^2+b^2)/2$, so that $p(t)=t^4-2ct^2+1$.  Let $P_k:=P_k(c)$ be the polynomial in $c$ satisfy the recursion relation
$$
(6+2k)P_k(c)
=4k cP_{k-2}(c)-2(k-3)P_{k-4}(c)
$$
for $k\geq 0$.
Then set
$$
P(c,z):=\sum_{k\geq -4}P_k(c)z^{k+4}=\sum_{k\geq 0}P_{k-4}(c)z^{k}.
$$
so that after some straightforward rearrangement of terms we have

\begin{align*}
0&=\sum_{k\geq 0}(6+2k)P_k(c)z^k
-4c\sum_{k\geq 0}kP_{k-2}(c)z^{k} +2\sum_{k\geq 0}(k-3)P_{k-4}(c)z^{k}  \\
&=(-2z^{-4} +8cz^{-2}-6)P(c,z) +(2z^{-3}-4cz^{-1}+2z)\frac{d}{dz}P(c,z)  \\
&\quad+(2z^{-4}-8cz^{-2})P_{-4}(c)  -4cP_{-3}(c)  z^{-1} -2P_{-2}(c)z^{-2} -4P_{-1}(c)z^{-1}.
\end{align*}

Hence $P(c,z)$ must satisfy the differential equation
\begin{equation}\label{funde}
\frac{d}{dz}P(c,z)-\frac{3z^4-4c z^2+1}{z^5-2cz^3+z}P(c,z)=\frac{2\left(P_{-1}+cP_{-3} \right)z^3 +P_{-2} z^2+(4cz^2-1)P_{-4} }{z^5-2cz^3+z}
\end{equation}
This has integrating factor
\begin{align*}
\mu(z)&
=\exp \int\left( \frac{-2 \left(z^3-cz\right)}{1-2 c z^2+z^4 }-\frac{1}{z}\right)\,dz  \\
&=\exp(-\frac{1}{2} \ln(1-2 c z^2+z^4)-\ln (z))=\frac{1}{z\sqrt{1-2 c z^2+z^4}}.
\end{align*}

\subsection{Elliptic Case 1}
If we take initial conditions $P_{-3}(c)=P_{-2}(c)=P_{-1}(c)=0$ and $P_{-4}(c)=1$ then we arrive at a generating function
$$
P_{-4}(c,z):=\sum_{k\geq -4}P_{-4,k}(c)z^{k+4}=\sum_{k\geq 0}P_{-4,k-4}(c)z^{k},
$$
defined in terms of an elliptic integral
\begin{align*}
P_{-4}(c,z)&=z\sqrt{1-2 c z^2+z^4}\int \frac{4cz^2-1}{z^2(z^4-2c z^2+1)^{3/2}}\, dz.
\end{align*}

\subsection{Elliptic Case 2}
If we take initial conditions $P_{-4}(c)=P_{-3}(c)=P_{-1}(c)=0$ and $P_{-2}(c)=1$, we arrive at a generating function defined in terms of another elliptic integral:
\begin{align*}
P_{-2}(c,z)&=z\sqrt{1-2 c z^2+z^4}\int \frac{1}{ (z^4-2c z^2+1)^{3/2}}\, dz.
\end{align*}

\subsection{Gegenbauer  Case 3}
If we take $P_{-1}(c)=1$, and $P_{-2}(c)=P_{-3}(c)=P_{-4}(c)=0$ and set $\displaystyle{P_{-1}(c,z)=\sum_{n\geq 0}P_{-1,n-4}z^n}$, then
$$
P_{-1}(c,z)=\frac{1}{c^2-1}\left(cz-z^3-cz+c^2z^3-\sum_{k=2}^\infty c Q_n^{(-1/2)}(c)z^{2n+1}\right),
$$
where $Q^{(-1/2)}_n(c)$ is the $n$-th Gegenbauer polynomial.   Hence
\begin{align*}
P_{-1,-4}(c)&=P_{-1,-3}(c)=P_{-1,-2}(c) =P_{-1,2m}(c)=0, \\
P_{-1,-1}(c)&=1,  \\
P_{-1,2n-3}(c)&=\frac{-cQ_{n}(c)}{c^2-1},
\end{align*}
for $m\geq 0$ and $n\geq 2$ .
The $Q^{(-1/2)}_n(c)$ are known to satisfy the second order differential equation:
\begin{align*}
(1-c^2)\frac{d^2}{d^2 c}Q^{(-1/2)}_n(c)+n(n-1)Q^{(-1/2)}_{n}(c)=0
\end{align*}
so that the $P_{-1,k}:=P_{-1,k}(c)$ satisfy the second order differential equation
\begin{align*}
(c^4-c^2)\frac{d^2}{d^2 c}P_{-1,2n-3}+2c(c^2+1)\frac{d}{d  c}P_{-1,2n-3}+(-c^2n(n-1)-2)P_{-1,2n-3}=0
\end{align*}
for $n\geq 2$.

\subsection{Gegenbauer Case 4}
Next we consider the initial conditions $P_{-1}(c)=0=P_{-2}(c)=P_{-4}(c)=0$ with $P_{-3}(c)=1$ and set
$$
P_{-3}(c,z)=\sum_{n\geq 0}P_{-3,n-4}(c)z^n=\frac{1}{c^2-1}\left(c^2z-cz^3-z+cz^3-\sum_{k=2}^\infty Q_n^{(-1/2)}(c)z^{2n+1}\right),$$
where $Q^{(-1/2)}_n(c)$ is the $n$-th Gegenbauer polynomial.  Hence
\begin{align*}
P_{-3,-4}(c)&=P_{-3,-2}(c)=P_{-3,-1}(c) =P_{-1,2m}(c)=0, \\
P_{-3,-3}(c)&=1,  \\
P_{-3,2n-3}(c)&=\frac{-Q_{n}(c)}{c^2-1},
\end{align*}
for $m\geq 0$ and $n\geq 2$ and hence \begin{align*}
(c^2-1)\frac{d^2}{d^2 c}P_{-3,2n-3}+4c \frac{d}{d c}P_{-3,2n-3} -(n+1)(n-2)P_{-3,2n-3} =0
\end{align*}
for $n\geq 2$ and $P_{-1,2n-3}=cP_{-3,2n-3}$ for $n\geq 2$.

We'll see in the last section how the polynomials described in the elliptic cases 1 and 2 are particular examples of associated ultraspherical polynomials.

The importance of these families of polynomials come from our previous work describing the universal central extension of the DJKM algebra:
\begin{thm}[\cite{MR2813377}] Let $\mathfrak g$ be a simple finite dimensional Lie algebra over the complex numbers with   the Killing form $(\,|\,)$ and define $\psi_{ij}(c)\in\Omega_R^1/dR$ by
\begin{equation}
\psi_{ij}(c)=\begin{cases}
\omega_{i+j-2}&\quad \text{ for }\quad i+j=1,0,-1,-2 \\
P_{-3,i+j-2}(c) (\omega_{-3}+c\omega_{-1})&\quad \text{for} \quad i+j =2n-1\geq 3,\enspace n\in\mathbb Z, \\
P_{-3,i+j-2}(c) (c\omega_{-3}+\omega_{-1})&\quad \text{for} \quad i+j =-2n+1\leq - 3, n\in\mathbb Z, \\
P_{-4,|i+j|-2}(c) \omega_{-4} +P_{-2,|i+j|-2}(c)\omega_{-2}&\quad\text{for}\quad |i+j| =2n \geq 2, n\in\mathbb Z. \\
\end{cases}
\end{equation}
The universal central extension of the Date-Jimbo-Kashiwara-Miwa  algebra is the $\mathbb Z_2$-graded Lie algebra
$$
\widehat{\mathfrak g}=\widehat{\mathfrak g}^0\oplus \widehat{\mathfrak g}^1,
$$
where
$$
\widehat{\mathfrak g}^0=\left(\mathfrak g\otimes \mathbb C[t,t^{-1}]\right)\oplus \mathbb C\omega_{0},\qquad \widehat{\mathfrak g}^1=\left(\mathfrak g\otimes \mathbb C[t,t^{-1}]u\right)\oplus \mathbb C\omega_{-4}\oplus \mathbb C\omega_{-3}\oplus \mathbb C\omega_{-2}\oplus \mathbb C\omega_{-1}
$$
with bracket
\begin{align*}
[x\otimes t^i,y\otimes t^j]&=[x,y]\otimes t^{i+j}+\delta_{i+j,0}j(x,y)\omega_0, \\ \\
[x\otimes t^{i-1}u,y\otimes t^{j-1}u]&=[x,y]\otimes (t^{i+j+2}-2ct^{i+j}+t^{i+j-2}) \\
 &\hskip 40pt+\left(\delta_{i+j,-2}(j+1) -2cj\delta_{i+j,0} +(j-1)\delta_{i+j,2}\right)(x,y)\omega_0, \\ \\
[x\otimes t^{i-1}u,y\otimes t^{j}]&=[x,y]u\otimes t^{i+j-1}+ j(x,y)\psi_{ij}(c).
\end{align*}

\end{thm}

\section{Differential equations for Elliptic type 1 and 2}

\subsection{Elliptic type 1}
How we arrive at the fourth order linear differential equation that the polynomials $P_{-4,n}$ satisfy stems from the approach used in Afken's book \cite{MR1810939} for finding the second order linear differential equation that Legendre polynomials satisfy using only the recursion relation they satisfy.  His technique seems to be indirectly based on ideas used in the theory of Gr\"obner basis.    Unfortunately the calculations and relations involved are rather tedious.

From now on we are going to reindex the polynomials $P_{-4,n}$:
\begin{align*}
P_{-4}(c,z)&=z\sqrt{1-2 c z^2+z^4}\int \frac{4cz^2-1}{z^2(z^4-2c z^2+1)^{3/2}}\, dz=\sum_{n=0}^\infty P_{-4,n}(c)z^n \\
&=1+z^4+\frac{4c}{5}z^6 +\frac{1}{35} \left(32 c^2-5\right) z^8+\frac{16}{105} c \left(8 c^2-3\right) z^{10} \\
 &\quad -\frac{\left(2048 c^4-1248 c^2+75\right) }{1155}z^{12}+O(z^{14})
\end{align*}
This means that now $P_{-4,0}(c)=1$, $P_{-4,1}(c)=P_{-4,2}(c)=
P_{-4.3}(c)=0$.
Besides $P_{-4,0}(c)$,  the first few nonzero polynomials in $c$ are
$$
P_{-4,4}(c)=1,\quad P_{-4,6}=\frac{4c}{5}, \quad P_{-4,8}=
\frac{32c^2-5}{35}
$$
$$
 P_{-4,10}=\frac{16}{105} c \left(8 c^2-3\right),\quad P_{-4,12}=-\frac{\left(2048 c^4-1248 c^2+75\right) }{1155}
$$
and
$P_{-4,n}(c)$ satisfy the following recursion:
\begin{equation}
(6+2k)P_{k+4}(c)
=4k cP_{k+2}(c)-2(k-3)P_{k}(c).
\end{equation}

Our goal in this section is to find families of linear differential equation in $c$ that these polynomials satisfy.
We start off with the generating function
\begin{align*}
P_{-4}(c,z)&=z\sqrt{1-2 c z^2+z^4}\int \frac{4cz^2-1}{z^2(z^4-2c z^2+1)^{3/2}}\, dz \\
&=z\sqrt{1-2 c z^2+z^4}\left(  \sum_{n=0}^\infty  \frac{4cQ_n^{(3/2)}(c)}{2n+1} z^{2n+1}- \sum_{n=0}^\infty  \frac{Q_n^{(3/2)}(c)}{2n-1} z^{2n-1}  \right)
\end{align*}
where $Q_n^{(\lambda)}(c)$ is the $n$-Gegenbauer polynomial.
These polynomials satisfy the second order linear ODE:
\begin{align*}
(1-c^2)y''-(2\lambda+1)cy'+n(n+2\lambda)y=0
\end{align*}
where the derivative is with respect to $c$.  Thus for $\lambda =3/2$ we get
\begin{equation}\label{2ndODE}
(1-c^2)(Q_n^{(3/2)})''(c)-4c(Q_n^{(3/2)})'(c)+n(n+3)Q_n^{(3/2)}(c)=0.
\end{equation}
Rewrite the expansion formula for $P_{-4}(c,z)$ to get
\begin{equation}\label{2nODE.1}
 z^{-1}(1-2 c z^2+z^4)^{-1/2}P_{-4}(c,z)=  \sum_{n=0}^\infty  \frac{4cQ_n^{(3/2)}(c)}{2n+1} z^{2n+1}- \sum_{n=0}^\infty  \frac{Q_n^{(3/2)}(c)}{2n-1} z^{2n-1}   ,
\end{equation}
and apply the differential operator $\displaystyle{L:=(1-c^2)\frac{d^2}{dc^2}-4c\frac{d}{dc}}$ to
 the right hand side to get
\begin{align*}
 L\left( 4cQ_n^{(3/2)}(c)  \right)  &=\left((1-c^2)\frac{d^2}{dc^2}-4c\frac{d}{dc}\right)(4cQ_n^{(3/2)}(c)) \\
&=-8(n+1)Q_{n+1}^{(3/2)}(c)- 4c (n^2+n-2)Q_n^{(3/2)}(c) \\
 \end{align*}
 using the identity
 \begin{equation}\label{wikiref}
 (1-c^2)\frac{d}{dc}\left(Q_n^{(\lambda)}(c)\right)=(n+2\lambda)cQ_n^{(\lambda)}(c)-(n+1)Q_{n+1}^{(\lambda)}(c) .
 \end{equation}
 Using a number of simplifications that we have put on http://arxiv.org/archive/math we arrive at 
After some simplification we get
\begin{align*}
L\left(\sum_{n=0}^\infty \frac{4cQ_n^{(3/2)}(c)}{2n+1} z^{2n+1}\right)
&= \sum_{n=0}^\infty \frac{-8(n+1)Q_{n+1}^{(3/2)}(c)- 4c (n^2+n-2)Q_n^{(3/2)}(c) }{2n+1} z^{2n+1}   \\ \\
&=-4 \frac{1}{z(z^4-2cz^2+1)^{3/2}}-4 \int\frac{1}{z^2(z^4-2cz^2+1)^{3/2}}\,dz    \\
&\quad - cz^2\frac{d}{dz}\left(\frac{1}{(z^4-2cz^2+1)^{3/2}}\right)  -\frac{cz}{(z^4-2cz^2+1)^{3/2}}\\
&\quad+ 9 c\int\frac{1}{(z^4-2cz^2+1)^{3/2}}\,dz.
\end{align*}
In addition we have
\begin{align*}
L\left( \sum_{n=0}^\infty  \frac{Q_n^{(3/2)}(c)}{2n-1} z^{2n-1} \right)&=-\sum_{n=0}^\infty  \frac{n(n+3)Q_n^{(3/2)}(c)}{2n-1} z^{2n-1}   \\
&=-\left(\frac{1}{4}z^2\frac{d^2}{dz^2}+\frac{9}{4}z\frac{d}{dz}+\frac{7}{4}\right)\sum_{n=0}^\infty  \frac{Q_n^{(3/2)}(c)}{2n-1} z^{2n-1}  \\
&=-\frac{1}{4}z^2\frac{d}{dz}\left( \frac{1}{z^2(z^4-2c z^2+1)^{3/2}}\right)  - \frac{9}{4z(z^4-2c z^2+1)^{3/2}} \\
&\quad -\frac{7}{4}\int\frac{1}{z^2(z^4-2c z^2+1)^{3/2}}\,dz.
\end{align*}
Thus the right hand side of \eqnref{2nODE.1} becomes
\begin{align*}
&\quad -4 \frac{1}{z(z^4-2cz^2+1)^{3/2}}-4 \int\frac{1}{z^2(z^4-2c z^2+1)^{3/2}}\,dz    \\
&\quad - cz^2\frac{d}{dz}\left(\frac{1}{(z^4-2cz^2+1)^{3/2}}\right)  -\frac{cz}{(z^4-2cz^2+1)^{3/2}}\\
&\quad+ 9 c\int\frac{1}{(z^4-2c z^2+1)^{3/2}}\,dz \\
&\quad +\frac{1}{4}z^2\frac{d}{dz}\left( \frac{1}{z^2(z^4-2c z^2+1)^{3/2}}\right)  + \frac{9}{4z(z^4-2c z^2+1)^{3/2}} \\
&\quad +\frac{7}{4}\int\frac{1}{z^2(z^4-2c z^2+1)^{3/2}}\,dz  \\  \\
&= - \frac{7}{4z(z^4-2cz^2+1)^{3/2}}    + \frac{6c z^2\left(z^2-c\right)}{\left(z^4-2 c z^2+1\right)^{5/2}}\\
&\quad  -\frac{cz}{(z^4-2cz^2+1)^{3/2}}-\frac{2 z^2\left(4 z^4-5 c z^2+1\right)}{z^3 \left(-2 c z^2+z^4+1\right)^{5/2}}    \\
&\quad+\frac{9}{4} \frac{1}{z\sqrt{z^4-2cz^2+1}}P_{-4}(c,z) .
\end{align*}

 Applying the differential operator $L$ to
 the left hand side of \eqnref{2nODE.1}, we get
\begin{align*}
L&\left(z^{-1}(1-2 c z^2+z^4)^{-1/2}P_{-4}(c,z)\right)   \\
&= (1-c^2)\frac{d^2}{dc^2}\left(z^{-1}(1-2 c z^2+z^4)^{-1/2}P_{-4}(c,z)\right) \\
&\quad -4c\frac{d}{dc}\left(z^{-1}(1-2 c z^2+z^4)^{-1/2}P_{-4}(c,z)\right)  \\
&= \frac{z(-4c+(3+5c^2)z^2-4cz^4)}{ (1-2 c z^2+z^4)^{5/2}} P_{-4}(c,z) \\
&\quad + \frac{ -4c+(2+6c^2)z^2-4cz^4}{z(1-2 c z^2+z^4)^{3/2}}\frac{d}{dc}\left(P_{-4}(c,z)\right) \\
&\quad +\frac{(1-c^2)}{z(1-2 c z^2+z^4)^{1/2}}\frac{d^2}{dc^2}\left(P_{-4}(c,z)\right).
\end{align*}
Hence we have
\begin{align*}
&\frac{z(-4c+(3+5c^2)z^2-4cz^4)}{ (1-2 c z^2+z^4)^{5/2}} P_{-4}(c,z)  + \frac{ -4c+(2+6c^2)z^2-4cz^4}{z(1-2 c z^2+z^4)^{3/2}}\frac{d}{dc}\left(P_{-4}(c,z)\right) \\
&  +\frac{(1-c^2)}{z(1-2 c z^2+z^4)^{1/2}}\frac{d^2}{dc^2}\left(P_{-4}(c,z)\right) \\ \\
&=  \frac{-4cz^2-7}{4z(z^4-2cz^2+1)^{3/2}}   - cz^2\frac{d}{dz}\left(\frac{1}{(z^4-2cz^2+1)^{3/2}}\right) \\
&\quad +\frac{1}{4}z^2\frac{d}{dz}\left( \frac{1}{z^2(z^4-2c z^2+1)^{3/2}}\right)  +\frac{9}{4} c\frac{1}{z\sqrt{z^4-2cz^2+1}}P_{-4}(c,z)
\end{align*}
as
\begin{align*}
\frac{d}{dz}\left(z^{-2}(z^4-2c z^2+1)^{-3/2}\right)
&=\frac{-2(4z^4-5c z^2+1)}{z^{3}(z^4-2c z^2+1)^{5/2}}.
\end{align*}
As a consequence
\begin{align*}
&\frac{z(-4c+(3+5c^2)z^2-4cz^4)}{ (1-2 c z^2+z^4)^{5/2}} P_{-4}(c,z)  + \frac{ -4c+(2+6c^2)z^2-4cz^4}{z(1-2 c z^2+z^4)^{3/2}}\frac{d}{dc}\left(P_{-4}(c,z)\right) \\
&  +\frac{(1-c^2)}{z(1-2 c z^2+z^4)^{1/2}}\frac{d^2}{dc^2}\left(P_{-4}(c,z)\right) \\ \\
&= - \frac{4cz^2+7}{4z(z^4-2cz^2+1)^{3/2}}   - c\left(\frac{6 z^3 \left(c-z^2\right)}{\left(z^4-2 c z^2+1\right)^{5/2}}\right) \\
&\quad -\frac{2 \left(4 z^4-5 c z^2+1\right)}{z \left(z^4-2 c z^2+1\right)^{5/2}}     \\
&\quad+\frac{9}{4} \frac{1}{z\sqrt{z^4-2cz^2+1}}P_{-4}(c,z)
 \end{align*}
%
%
which gives us
\begin{align*}
-\frac{9}{4}&+5cz^2-\left(\frac{15}{4}+4c^2\right)z^4+5cz^6 \\
&=\left(-\frac{9}{4} (z^4-2cz^2+1)^2   +z^2(-4c+(3+5c^2)z^2-4cz^4)\right)P_{-4}(c,z)  \\
&\quad  + (-4c+(2+6c^2)z^2-4cz^4)(z^4-2cz^2+1)\frac{d}{dc}\left(P_{-4}(c,z)\right) \\
& \quad  +(1-c^2)(z^4-2cz^2+1)^2\frac{d^2}{dc^2}\left(P_{-4}(c,z)\right).
\end{align*}
Expanding this out in detail
and writing $P_{-4,k}(c)$ as $P_k$ we will obtain

 \begin{align}
0&=-9 P_{n}+20cP_{n-2}-\left(16 c^2+6\right)  P_{n-4}  + 20cP_{ n-6}   -9 P_{ n-8}  \label{longequationelliptic1}  \\
&\hskip 40pt -16c P'_{ n} + 8(1+7c^2)P'_{ n-2} -48c(1+c^2)P'_{ n-4}    +8(1+7c^2)P'_{ n-6}   -16c P'_{ n-8}  \notag \\
&\hskip 40pt + 4(1-c^2)\left(P_{ n}''  - 4 c  P''_{ n-2}  +  2(1+2c^2) P''_{ n-4}- 4c  P''_{ n-6}  +P''_{ n-8} \right).\notag
\end{align}

We now differentiate with respect to $c$ the recursion relations
\begin{equation}\label{recursionelliptic1}
(6+2k)P_{k+4}
=4k cP_{k+2} -2(k-3)P_{k}
\end{equation}
 to get
\begin{align}
(6+2k)P_{k+4}'
&=4k P_{k+2}+4kcP'_{k+2}-2(k-3)P_{k}'\label{1storderelliptic1} \\
(6+2k)P_{k+4}''
&=8k P'_{k+2} +4kcP''_{k+2}-2(k-3)P_{k}'' .\label{2ndorderelliptic1}
\end{align}
After setting $k=n-8$ in the last equation we get
\begin{align}
0=-(2n-10)P_{n-4}''   +8(n-8) P'_{n-6}+4(n-8)cP''_{n-6}-2(n-11)P_{n-8}'' .
\end{align}
So multiplying \eqnref{longequationelliptic1} by $(n-11)$ and adding it to $2(1-c^2)$ times the above gives us
\begin{align}
0\label{longequationelliptic1.2}
&=   (n-11)\Big(   -9 P_{n}+20cP_{n-2}-\left(16 c^2+6\right)  P_{n-4}  + 20cP_{ n-6}   -9 P_{ n-8}  \notag\\
&\hskip 40pt -16c P'_{ n} + 8(1+7c^2)P'_{ n-2} -48c(1+c^2)P'_{ n-4}        -16c P'_{ n-8}  \notag \\
&\hskip 40pt + 4(1-c^2)\left(P_{ n}''  - 4 c  P''_{ n-2}     \right)  \Big) \notag\\
&\quad  +8 \left( c^2 (5 n-61)+  (3 n-27)\right) P'_{n-6}-  4(1-c^2)(4c^2(n-11)+n-17) P''_{n-4}   \notag  \\
&\quad +8(n-14)c(1-c^2)P''_{n-6}.
\end{align}

Setting $k=n-8$ in \eqnref{1storderelliptic1} we get
\begin{equation}
0=-(2n-10)P_{n-4}'
+4(n-8)P_{n-6}+4(n-8)cP'_{n-6}-2(n-11)P_{n-8}'
\end{equation}
Multiplying this equation by $-8c$ and add to the previous equation we get
\begin{align}
0
&=    (n-11)\Big(   -9 P_{n}+20cP_{n-2}-\left(16 c^2+6\right)  P_{n-4}     -9 P_{ n-8}  \notag\\
&\hskip 40pt   -16c P'_{ n} + 8(1+7c^2)P'_{ n-2} + 4(1-c^2)\left(P_{ n}''  - 4 c  P''_{ n-2}     \right)  \Big) \notag\\
&\quad+16 c \left(-3 \left(c^2+1\right) (n-11)+n-5\right)P_{n-4}'   -  4(1-c^2)(4c^2(n-11)+n-17) P''_{n-4}   \notag  \\
&\quad -12c(n-3)P_{ n-6} \notag    +8 \left(c^2 (n-29)+3 (n-9)\right)P'_{n-6}  +8(n-14)c(1-c^2)P''_{n-6}. \notag
\end{align}

Finally if we multiply the previous equation by 2 and add it to $-9$ times \eqnref{recursionelliptic1} we get
\begin{align}
0
&=  2  (n-11)\Big(   -9 P_{n}+20cP_{n-2}   -16c P'_{ n} + 8(1+7c^2)P'_{ n-2} + 4(1-c^2)\left(P_{ n}''  - 4 c  P''_{ n-2}     \right)  \Big) \notag\\
&\quad+\left(6 (n+7)-32 c^2 (n-11)\right)P_{n-4}    +32 c \left(-3 \left(c^2+1\right) (n-11)+n-5\right)P_{n-4}'   \notag \\
&\quad  -  8(1-c^2)(4c^2(n-11)+n-17) P''_{n-4}   \notag  \\
&\quad -60c(n-6)P_{ n-6} \notag    +16 \left(c^2 (n-29)+3 (n-9)\right)P'_{n-6}  +16(n-14)c(1-c^2)P''_{n-6}.
\notag 
\end{align}

This is an equation without the $P_{n-8}$ in it.   We now get rid of the $P_{n-6}$'s in them.

After setting $k=n-6$ in \eqnref{2ndorderelliptic1} we get
\begin{align}
0=-(2n-6)P_{n-2}''    +8(n-6) P'_{n-4} +4(n-6)cP''_{n-4} -2(n-9)P_{n-6}''   .
\end{align}
Now we multiply the previous equation by $-8(n-14)c(1-c^2)$ and add it to $n-9$ times the equation before it, we obtain
\begin{align}
0
&=   2  (n-11)(n-9)\Big(   -9 P_{n}+20cP_{n-2}   -16c P'_{ n} + 8(1+7c^2)P'_{ n-2} + 4(1-c^2) P_{ n}''     \Big) \notag\\
&\quad+(n-9)\left(6 (n+7)-32 c^2 (n-11)\right)P_{n-4}        \notag \\
&\quad +16 c \left(c^2-1\right) \left(n^2-23 n+156\right)P_{n-2}''   -32 c \left(c^2 \left(n^2-20 n+129\right)+4 n^2-86 n+420\right)P'_{n-4}  \notag \\
&\quad -8 \left(c^2-1\right) \left(60 c^2+n^2-26 n+153\right)P''_{n-4} \notag \\
&\quad -60c(n-9)(n-6)P_{ n-6} \notag    +16(n-9) \left(c^2 (n-29)+3 (n-9)\right)P'_{n-6}     \notag
\end{align}
From \eqnref{1storderelliptic1} with $k=n-6$ one has
\begin{align}
0=-(2n-6)P_{n-2}'   +4(n-6) P_{n-4}+4(n-6)cP'_{n-4}-2(n-9)P_{n-6}' .
\end{align}

We multiply this equation by $8 \left(c^2 ( n-29)+3 (n-9)\right)$ and add it to the previous equation to get
\begin{align}
0
&=     2  (n-11)(n-9)\Big(   -9 P_{n}+20cP_{n-2}   -16c P'_{ n} + 8(1+7c^2)P'_{ n-2} + 4(1-c^2) P_{ n}''     \Big) \notag\\
&\quad - 8(2n-6) \left(c^2 ( n-29)+3 (n-9)\right)P_{n-2}'   +16 c \left(c^2-1\right) \left(n^2-23 n+156\right)P_{n-2}''     \notag \\
&\quad  -6 \left(80 c^2 (n-5)-17 n^2+242 n-801\right)P_{n-4}\notag  -32 c \left(15 c^2 (n-3)+n^2-41 n+258\right)P'_{n-4} \notag \\
&\quad -8 \left(c^2-1\right) \left(60 c^2+n^2-26 n+153\right)P''_{n-4} \notag \\
&\quad -60c(n-9)(n-6)P_{ n-6} \notag
\end{align}

From \eqnref{recursionelliptic1} with $k=n-6$ one has
\begin{align*}
0=-(2n-6)P_{n-2}  +4(n-6) cP_{n-4} -2(n-9)P_{n-6}
\end{align*}
We multiply this equation by $-30c(n-6)$ and add it to the previous equation to get
\begin{align}
0
&=     2  (n-11)(n-9)\Big(   -9 P_{n}   -16c P'_{ n} + 8(1+7c^2)P'_{ n-2} + 4(1-c^2) P_{ n}''     \Big) \notag\\
&\quad +20 c \left(5 n^2-67 n+252\right)P_{n-2}  - 8(2n-6) \left(c^2 ( n-29)+3 (n-9)\right)P_{n-2}'   \notag  \\
&\quad +16 c \left(c^2-1\right) \left(n^2-23 n+156\right)P_{n-2}''     \notag \\
&\quad  +6 \left(-20 c^2 (n-4)^2+17 n^2-242 n+801\right)P_{n-4}\notag  -32 c \left(15 c^2 (n-3)+n^2-41 n+258\right)P'_{n-4} \notag \\
&\quad -8 \left(c^2-1\right) \left(60 c^2+n^2-26 n+153\right)P''_{n-4} \notag
\end{align}
The above equation now does not have the index $n-6$ in it.  Now
we want to eliminate the indices with $n-4$ in them.

From \eqnref{2ndorderelliptic1} with $k=n-4$ one has
\begin{align*}
0=-(2n-2)P_{n}''   +8(n-4) P'_{n-2} +4(n-4)cP''_{n-2}-2(n-7)P_{n-4}'' .
\end{align*}
We multiply this equation by $-4  (c^2-1 )  (60 c^2+n^2-26 n+153 ) $ and add it to $n-7$ times the previous equation to get
\begin{align}
0
&=      2  (n-11)(n-9)(n-7)\Big(   -9 P_{n}   -16c P'_{ n}    \Big) \notag\\
&\quad +20 c(n-7) \left(5 n^2-67 n+252\right)P_{n-2}      \notag  \\
&\quad  +6 (n-7)\left(-20 c^2 (n-4)^2+17 n^2-242 n+801\right)P_{n-4}\notag   \\
&\quad -32 c(n-7) \left(15 c^2 (n-3)+n^2-41 n+258\right)P'_{n-4} \notag \\
&\quad+480 \left(c^2-1\right) \left(c^2 (n-1)-n+9\right)P_{n}''  \notag\\
&\quad-32 \left(60 c^4 (n-4)+c^2 \left(-2 n^3+45 n^2-484 n+1749\right)+15 \left(n^2-14 n+45\right)\right)P'_{n-2}\notag  \\
&\quad  -960 c \left(c^2-1\right) \left(c^2 (n-4)-n+8\right)P''_{n-2}  \notag
\end{align}

From \eqnref{1storderelliptic1} with $k=n-4$ one has
\begin{align*}
0=-(2n-2)P_{n}'   +4(n-4) P_{n-2}+4(n-4)cP'_{n-2}-2(n-7)P_{n-4}'.
\end{align*}
We multiply this equation by $-16 c   (15 c^2 (n-3)+n^2-41 n+258 ) $ and add it to the previous equation to get
\begin{align}
0&=      -18  (n-11)(n-9)(n-7) P_{n}    \notag\\
&\quad +480 c \left(c^2 \left(n^2-4 n+3\right)-n^2+4 n+29\right)P_{n}' \notag   \\
&\quad+480 \left(c^2-1\right) \left(c^2 (n-1)-n+9\right)P_{n}''  \notag\\
&\quad +12 c \left(-80 c^2 \left(n^2-7 n+12\right)+3 n^3+70 n^2-1049 n+2564\right)P_{n-2}      \notag  \\
&\quad -480 \left(2 c^4 \left(n^2-5 n+4\right)-3 c^2 \left(n^2-8 n+7\right)+n^2-14 n+45\right)P'_{n-2} \notag   \\
&\quad  -960 c \left(c^2-1\right) \left(c^2 (n-4)-n+8\right)P''_{n-2}  \notag   \\
&\quad  +6 (n-7)\left(-20 c^2 (n-4)^2+17 n^2-242 n+801\right)P_{n-4}\notag
\end{align}

From \eqnref{recursionelliptic1} with $k=n-4$ one has
\begin{align*}
0=-2(n-1)P_{n}   +4(n-4)cP_{n-2}-2(n-7)P_{n-4}.
\end{align*}
We multiply this equation by $3\left(-20 c^2 (n-4)^2+17 n^2-242 n+801\right)$ and add it to the previous equation to get
\begin{align}
0
&=      120 (n-4)^2 \left(c^2 (n-1)-n+9\right) P_{n}    \notag\\
&\quad +480 c \left(c^2 \left(n^2-4 n+3\right)-n^2+4 n+29\right)P_{n}' \notag   \\
&\quad+480 \left(c^2-1\right) \left(c^2 (n-1)-n+9\right)P_{n}''  \notag\\
&\quad -240 c (n-2)^2 \left(c^2 (n-4)-n+8\right)P_{n-2}      \notag  \\
&\quad -480 \left(2 c^4 \left(n^2-5 n+4\right)-3 c^2 \left(n^2-8 n+7\right)+n^2-14 n+45\right)P'_{n-2} \notag   \\
&\quad  -960 c \left(c^2-1\right) \left(c^2 (n-4)-n+8\right)P''_{n-2}.  \notag
\end{align}
This can be rewritten as

\begin{align}
0&=     (n-4)^2 \left(c^2 (n-1)-n+9\right) P_{n}    \label{elliptic1eqn1}\\
&\quad +4 c \left(c^2 \left(n^2-4 n+3\right)-n^2+4 n+29\right)P_{n}' \notag   \\
&\quad+4 \left(c^2-1\right) \left(c^2 (n-1)-n+9\right)P_{n}''  \notag\\
&\quad -2 c (n-2)^2 \left(c^2 (n-4)-n+8\right)P_{n-2}      \notag  \\
&\quad -4 \left(2 c^4 \left(n^2-5 n+4\right)-3 c^2 \left(n^2-8 n+7\right)+n^2-14 n+45\right)P'_{n-2} \notag   \\
&\quad  -8 c \left(c^2-1\right) \left(c^2 (n-4)-n+8\right)P''_{n-2}.\notag
\end{align}
We have now reduced to a differential equation with only the
indices $n$ and $n-2$.   There is a bit of a trick to reduce it
down to a linear ODE with only the index $n$ in it.
The somewhat vague idea is
to find more equations to in order to cancel all but terms with
the index $n$ in them.

From \eqnref{2ndorderelliptic1} with $k=n-2$ one has
\begin{align*}
0=-2(n+1)P_{n+2}''   +8(n-2) P'_{n} +4(n-2)cP''_{n}-2(n-5)P_{n-2}''  .
\end{align*}

We multiply this equation by $ -4  c(c^2-1)(c^2(n-4)+8-n) $ and add it to $n-5$ times the previous equation to get
\begin{align}
0
&= -2(n-5) c  (n-2)^2 \left(c^2 (n-4)-n+8\right)P_{n-2}      \notag  \\
&\quad -4 (n-5) \left(2 c^4 \left(n^2-5 n+4\right)-3 c^2 \left(n^2-8 n+7\right)+n^2-14 n+45\right)P'_{n-2} \notag   \\
&\quad + (n-5)  ( (n-4)^2 \left(c^2 (n-1)-n+9\right) P_{n}    \notag\\
&\quad +4 c \left(-8 c^4 \left(n^2-6 n+8\right)+c^2 \left(n^3+7 n^2-105 n+177\right)-n^3+n^2+89 n-273\right)P'_{n} \notag \\
&\quad -4 \left(c^2-1\right) \left(4 c^4 \left(n^2-6 n+8\right)+c^2 \left(-5 n^2+46 n-69\right)+n^2-14 n+45\right)P''_{n} \notag \\
&\quad +8(n+1)c(c^2-1)(c^2(n-4)+8-n)(n+1)P_{n+2}''    \notag
\end{align}
From \eqnref{1storderelliptic1} with $k=n-2$ one has
\begin{align*}
0= -2(n+1)P_{n+2}'   +4(n-2) P_{n}+4(n-2)cP'_{n}-2(n-5)P_{n-2}'
\end{align*}

We multiply this equation by
\begin{equation*}
 -2 (2c^4 (n-4)(n-1)-3c^2(n-7)(n-1)+ (n-9)(n-5))
\end{equation*}
and add it to the previous equation to get

\begin{align}
0
&=  -2(n-5) c  (n-2)^2 \left(c^2 (n-4)-n+8\right)P_{n-2}      \notag  \\
&\quad + \left(-16 c^4 \left(n^3-7 n^2+14 n-8\right)+c^2 \left(n^4+10 n^3-171 n^2+416 n-256\right)-n^2 \left(n^2-14 n+45\right)
\right)P_{n}    \notag\\
&\quad-4 c (n+1) \left(4 c^4 \left(n^2-6 n+8\right)+c^2 \left(-7 n^2+60 n-93\right)+3 \left(n^2-12 n+31\right)\right)P'_{n} \notag \\
&\quad -4 \left(c^2-1\right) \left(4 c^4 \left(n^2-6 n+8\right)+c^2 \left(-5 n^2+46 n-69\right)+n^2-14 n+45\right)P''_{n} \notag \\
&\quad +4(2c^4 (n-4)(n-1)-3c^2(n-7)(n-1)+ (n-9)(n-5))(n+1)P_{n+2}'  \notag  \\
&\quad +8(n+1)c(c^2-1)(c^2(n-4)+8-n)(n+1)P_{n+2}''    \notag
\end{align}
Next we get rid of the $P_{n-2}$ term by multiplyfing \eqnref{recursionelliptic1} with $k=n-2$;
\begin{equation*}
0=-2(n+1)P_{n+2}   +4(n-2) cP_{n}-2(n-5)P_{n-2},
\end{equation*}

 by $ -c(n-2)^2 \left(c^2 (n-4)-n+8\right)$
and adding it to the previous equation:
\begin{align}
0
&=-n^2 \left(4 c^4 \left(n^2-6 n+8\right)+c^2 \left(-5 n^2+46 n-69\right)+n^2-14 n+45\right)P_{n}    \notag\\
&\quad-4 c (n+1) \left(4 c^4 \left(n^2-6 n+8\right)+c^2 \left(-7 n^2+60 n-93\right)+3 \left(n^2-12 n+31\right)\right)P'_{n} \notag \\
&\quad -4 \left(c^2-1\right) \left(4 c^4 \left(n^2-6 n+8\right)+c^2 \left(-5 n^2+46 n-69\right)+n^2-14 n+45\right)P''_{n} \notag \\
&\quad +2c(n+1)(n-2)^2 \left(c^2 (n-4)-n+8\right)P_{n+2}    \notag \\
&\quad +4(n+1)(2c^4 (n-4)(n-1)-3c^2(n-7)(n-1)+ (n-9)(n-5))P_{n+2}'  \notag  \\
&\quad +8(n+1)^2c(c^2-1)(c^2(n-4)+8-n) P_{n+2}''  \notag
\end{align}

Letting $n\mapsto n-2$ in the above equation we get
\begin{align}
0&=-(n-2)^2 \left(4 c^4 \left(n^2-10 n+24\right)+c^2 \left(-5 n^2+66 n-181\right)+n^2-18 n+77\right)P_{n-2}    \label{elliptic1eqn2}\\
&\quad-4 c (n-1) \left(4 c^4 \left(n^2-10 n+24\right)+c^2 \left(-7 n^2+88 n-241\right)+3 \left(n^2-16 n+59\right)\right)P'_{n-2} \notag \\
&\quad -4 \left(c^2-1\right) \left(4 c^4 \left(n^2-10 n+24\right)+c^2 \left(-5 n^2+66 n-181\right)+n^2-18 n+77\right)P''_{n-2} \notag \\
&\quad +2c(n-1)(n-4)^2 \left(c^2 (n-6)-n+10\right)P_{n}    \notag \\
&\quad +4(n-1)(2c^4 (n-6)(n-3)-3c^2(n-9)(n-3)+ (n-11)(n-7))P_{n}'  \notag  \\
&\quad +8(n-1)^2c(c^2-1)(c^2(n-6)-n+10)P_{n}''  \notag
\end{align}

We compare this to \eqnref{elliptic1eqn1}
\begin{align}
0&= -2 c (n-2)^2 \left(c^2 (n-4)-n+8\right)P_{n-2}      \label{elliptic1eqn3}  \\
&\quad -4 \left(2 c^4 \left(n^2-5 n+4\right)-3 c^2 \left(n^2-8 n+7\right)+n^2-14 n+45\right)P'_{n-2} \notag   \\
&\quad  -8 c \left(c^2-1\right) \left(c^2 (n-4)-n+8\right)P''_{n-2} \notag  \\
&\quad +(n-4)^2 \left(c^2 (n-1)-n+9\right) P_{n}   \notag \\
&\quad +4 c \left(c^2 \left(n^2-4 n+3\right)-n^2+4 n+29\right)P_{n}' \notag   \\
&\quad+4 \left(c^2-1\right) \left(c^2 (n-1)-n+9\right)P_{n}''  \notag
\end{align}
Our goal is to eliminate the $P_{n-2}''  $ term, then the $P_{n-2}'  $ and finally the $P_{n-2}$ to arrive at a differential equation with just the derivatives of $P_{n}$ in it.   We first have to lower the degree of $c$ in the polynomial in front of $P_{n-2}''  $.
 Thus if we multiply \eqnref{elliptic1eqn3} by $ -2c(n-6)$
 and add it to \eqnref{elliptic1eqn2} we get
\begin{align}
0
&=(n-11) (n-2)^2 \left(c^2 (n+1)-n+7\right)P_{n-2} \notag  \\
&\quad +4 c (n-11) \left(c^2 \left(n^2-1\right)-n^2+33\right))P_{n-2}'  \notag  \\
&\quad +4(n-11)  \left(c^2-1\right) \left(c^2 (n+1)-n+7\right)P_{n-2}''    \notag  \\
&\quad-8 c (n-11) (n-4)^2 \left(c^2 (n-7)-n-1\right) P_{n}   \notag\\
&\quad -4 (n-11) \left(c^2 \left(n^2-8 n+39\right)-n^2+8 n-7\right)P_{n}' \notag  \\
&\quad -32 c(n-11)  \left(c^2-1\right) P_{n}''   \notag
\end{align}

If $n\neq 11$ we get
\begin{align} \label{ellipticeqn3}
0&=  (n-2)^2 \left(c^2 (n+1)-n+7\right)P_{n-2}   \\
&\quad +4 c   \left(c^2 \left(n^2-1\right)-n^2+33\right))P_{n-2}'  \notag  \\
&\quad +4  \left(c^2-1\right) \left(c^2 (n+1)-n+7\right)P_{n-2}''    \notag  \\
&\quad-8 c  (n-4)^2  P_{n}   \notag\\
&\quad -4  \left(c^2 \left(n^2-8 n+39\right)-n^2+8 n-7\right)P_{n}' \notag  \\
&\quad -32 c  \left(c^2-1\right) P_{n}'' .  \notag
\end{align}
 We lower the degree of $c$ in the polynomial in front of $P_{n-2}''  $ another time by multiplying this last equation by $2c(n-4)$ and add it to $n+1$ times \eqnref{elliptic1eqn1} to get
\begin{align}
0
&= 8 c (n-9) (n-2)^2P_{n-2} \notag  \\
&\quad +4 (n-9) \left(c^2 \left(n^2-4 n+27\right)-n^2+4 n+5\right)P_{n-2}'  \notag  \\
&\quad +32(n-9)  c \left(c^2-1\right) P_{n-2}''    \notag  \\
&\quad+(n-9) (n-4)^2 \left(c^2 (n-7)-n-1\right) P_{n}   \notag\\
&\quad-4 c (n-9) \left(c^2 \left(n^2-12 n+35\right)-n^2+12 n-3\right)P_{n}' \notag  \\
&\quad +4(n-9) \left(c^2-1\right)  \left(c^2 (n-7)-n-1\right)P_{n}'' .  \notag
\end{align}
So that if $n\neq 9,11$ one has
\begin{align} \label{ellipticeqn4}
0&= 8 c  (n-2)^2P_{n-2}   \\
&\quad +4   \left(c^2 \left(n^2-4 n+27\right)-n^2+4 n+5\right)P_{n-2}'  \notag  \\
&\quad +32   c \left(c^2-1\right) P_{n-2}''    \notag  \\
&\quad+  (n-4)^2 \left(c^2 (n-7)-n-1\right) P_{n}   \notag\\
&\quad-4 c   \left(c^2 \left(n^2-12 n+35\right)-n^2+12 n-3\right)P_{n}' \notag  \\
&\quad +4 \left(c^2-1\right)  \left(c^2 (n-7)-n-1\right)P_{n}''    \notag
\end{align}

We lower the degree of $c$ in the polynomial in front of $P_{n-2}''  $ another time by multiplying this last equation by $c(n+1)$ and add it to $-8$ times \eqnref{ellipticeqn3} to get

\begin{align}
0
&= 8 (n-7) (n-2)^2P_{n-2} \notag  \\
&\quad +4 c (n-7) \left(c^2 \left(n^2-4 n-5\right)-n^2+4 n+37\right)P_{n-2}'  \notag  \\
&\quad +32 (c^2-1) (n-7)P_{n-2}''    \notag  \\
&\quad +c (n-7) (n-4)^2 \left(c^2 (n+1)-n-9\right)P_{n}   \notag\\
&\quad -4 (n-7) \left(c^4 \left(n^2-4 n-5\right)-c^2 \left(n^2+4 n-45\right)+8 (n-1)\right) P_{n}' \notag  \\
&\quad  +4  c (c^2-1) (n-7) \left(c^2 (n+1)-n-9\right)P_{n}''   \notag
\end{align}
Which if $n\neq 7,9,11$, then we have

\begin{align} \label{ellipticeqn5}
0&= 8   (n-2)^2P_{n-2}    \\
&\quad +4 c   \left(c^2 \left(n^2-4 n-5\right)-n^2+4 n+37\right)P_{n-2}'  \notag  \\
&\quad +32 (c^2-1)  P_{n-2}''    \notag  \\
&\quad +c   (n-4)^2 \left(c^2 (n+1)-n-9\right)P_{n}   \notag\\
&\quad -4   \left(c^4 \left(n^2-4 n-5\right)-c^2 \left(n^2+4 n-45\right)+8 (n-1)\right) P_{n}' \notag  \\
&\quad  +4  c (c^2-1)   \left(c^2 (n+1)-n-9\right)P_{n}''   \notag
\end{align}
 We now want to get rid of the term with $P_{n-2}''$ in it.  This is done by
multiplying the previous equation by $c$ and add it to $-1$ times \eqnref{ellipticeqn4} we get
\begin{align}
0
&= 4  (c^2-1 )^2  (n-5 )(n+1)P_{n-2}'  \notag  \\
&\quad +  (c^2-1 )^2 (n-4)^2 (n+1) P_{n}   \notag\\
&\quad -4 c (c^2-1 )^2  (n-5 )(n+1)P_{n}' \notag  \\
&\quad +4  (c^2-1 )^3 (n+1)P_{n}'' \notag
\end{align}
Thus as $c\neq \pm 1$ and we are assuming $n\neq  -1,7,9,11$ then we have
\begin{align} \label{elliptic1eqn6}
0&= 4    (n-5)P_{n-2}'  +   (n-4)^2  P_{n}    -4 c   (n-5  )P_{n}'   +4  (c^2-1 )  P_{n}''.
\end{align}
If we differentiate this with respect to $c$ we get
\begin{align}
0&= 4    (n-5)P_{n-2}''  +   (n-4)^2  P_{n} '   -4    (n-5  )P_{n}'  -4   c (n-5  )P_{n}'' +8c  P_{n}''+4  (c^2-1 )  P_{n}'''   \\
&= 4    (n-5)P_{n-2}''  +   (n-6)^2  P_{n} '     -4   c (n-7  ) P_{n}'' +4  (c^2-1 )  P_{n}'''. \notag
\end{align}
Now we work on the coefficient in front of $P_{n-2}'$ to eliminate it.

We can multiply the previous equation by $-8(c^2-1)$ and add to $n-5$ times \eqnref{ellipticeqn5} to give us

\begin{align}
0
&=  8   (n-5) (n-2)^2P_{n-2}  \notag  \\
&\quad +4 c    (n-5) (c^2  (n^2-4 n-5 )-n^2+4 n+37 ) P_{n-2}'  \notag  \\
&\quad +c   (n-5) (n-4)^2  (c^2 (n+1)-n-9 )P_{n}   \notag\\
&\quad+4  (c^4  (-(n-5)^2 ) (n+1)+c^2  (n^3-3 n^2-41 n+153 )-6 n^2+24 n+32 ) P_{n}' \notag  \\
&\quad +4 c ( c^2-1) (8 ( n-7) + (n-5) ( c^2 (1 + n)-9 - n))P_{n}''   \notag  \\
&\quad   -32 ( c^2-1)^2  P_{n}''' \notag
\end{align}

Now we multiply \eqnref{elliptic1eqn6} by $c(37 + 4 n - n^2 + c^2 (n-5)(n+1))$ and add it to $-1$ times the equation above to give us
\begin{align}
0
&= -8 (n - 5) (n - 2)^2 P_{n-2}   + 8 c (n-4)^2 ( n-1) P_{n}   \notag\\
&\quad  -8  (c^2-1 )  (3 n^2-12 n-16 )P_{n}'  +192 c  (c^2-1 )P_{n}''     +32 ( c^2-1)^2 P_{n}''' \notag
\end{align}
Thus if $c\neq \pm 1$, $n\neq -1,7,9,11$ we have

\begin{align}
0&=- (n - 5) (n - 2)^2 P_{n-2}  +  c (n-4)^2 ( n-1) P_{n}   \notag\\
&\quad  -  (c^2-1 )  (3 n^2-12 n-16 )P_{n}' +24 c  (c^2-1 )P_{n}''      +4 ( c^2-1)^2 P_{n}'''\notag
\end{align}

 If we differentiate this with respect to $c$ we get
 \begin{align}
0
&= - (n - 5) (n - 2)^2 P_{n-2}'  \notag  \\
&\quad +   (n-4)^2 ( n-1) P_{n}  \notag\\
&\quad  +  c (n^3-15n^2+48n+16)P_{n}' \notag  \\
&\quad +(-c^2(3n^2-12n-88)+3n^2-12n-40)P_{n}''   \notag \\
&\quad  +40c(c^2-1)P_{n}''' \notag \\
&\quad  +4 ( c^2-1)^2 P_{n}^{(iv)} \notag
\end{align}

Now we multiply this by $4$ and add it to $(n-2)^2$ times \eqnref{elliptic1eqn6} to get the $4$th order linear differential equation satisfied by the polynomials $P_n$. We proved
%
 
 \begin{thm}
 The polynomials $P_n=P_{-4,n}$ satisfy the following differential equation:
 \begin{align}\label{4thorder1}
&16 ( c^2-1)^2   P_{n}^{(iv)}+160 c  (c^2-1 )  P_{n}''' -8  (c^2  (n^2-4 n-46 )-n^2+4 n+22 )P_{n}''  \\
\quad -24 c ( n^2 - 4 n-6)P_{n}'  +(n-4)^2 n^2P_{n}=0.
\end{align}
\end{thm}
%
%

\subsection{Elliptic Case 2}
From now on we are going to reindex the polynomials $P_{-2,n}$:
\begin{align*}
P_{-2}(c,z)&=z\sqrt{1-2 c z^2+z^4}\int \frac{1}{ (z^4-2c z^2+1)^{3/2}}=\sum_{n=0}^\infty P_{-2,n}(c)z^n.\end{align*}
This means that now $P_{-2,2}(c)=1$, $P_{-2,3}(c)=P_{-2,1}(c)=
P_{-2,0}(c)=0$ and
$P_{-2,n}(c)$ satisfy the following recursion:

\begin{equation}
(6+2k)P_{k+4}(c)
=4k cP_{k+2}(c)-2(k-3)P_{k}(c)
\end{equation}

The first few nonzero polynomials in $c$ are
$$
P_{-2,2}(c)=1,\quad P_{-2,6}=1/5, \quad P_{-2,8}=8c/35,$$
$$
 P_{-2,10},=(-7+32c^2)/105,\quad P_{-2,12}=8c(-29+64c^2)/1155.
$$
so that

\begin{align*}
P_{-2}(c,z)&=z^2+\frac{1}{5}z^6+\frac{8c}{35}z^8+\frac{32c^2-7}{105}z^{10}+\frac{8c(64c^2-29)}{1155}z^{12}+O(z^{14}).
\end{align*}
After a very similar lengthy analysis as in the previous section we arrive at the following result:

\begin{thm}
The polynomials $P_n=P_{-2,n}$
satisfy the following fourth order linear differential equation 
\begin{align}\label{4thorder2}
  &16 ( c^2-1)^2   P_{n}^{(iv)}+160 c  (c^2-1 )  P_{n}'''    -8  (c^2  (n^2-4 n-42 )-n^2+4 n+18 )P_{n}''    \\
&\quad -24 c   (n^2-4 n-2 )P_{n}'  +(n-6) (n-2)^2   (n+2)P_{n}=0 .   \notag
\end{align}
\end{thm}
%

\section{Associated ultraspherical polynomials}

 After shifting the indices back by $4$ we obtain  that both families of the polynomials $P_{-4,n}(c)$ and $P_{-2,n}(c)$
 satisfy the recurrence relation
 
 \begin{equation}
  2k cP_{k-2}(c)= (3+k)P_{k}(c)     +(k-3)P_{k-4}(c).
\end{equation}

Note that all odd polynomials are zero. Set $k=2(n+1)$ and $q_{s}:=P_{2s}$. Then we have
\begin{equation}\label{qrecursion}
4(n+1)cq_n=(2n+5)q_{n+1}+ (2n-1)q_{n-1},
\end{equation}
 where $q_s=q_s(c)$.  

 For $q_s(c)=P_{-4,2s}(c)$ one has
\vskip 5pt
\noindent
\fbox{\begin{minipage}{5in}
\begin{gather}\label{q-4}
q_{-2}=P_{-4,-4}=1,\quad q_{-1}=P_{-4,-2}=0,\quad q_0=P_{-4,0}=1,  \\ q_1=P_{-4,2}=4c/5,\quad
 q_2=P_{-4,4}=\frac{32c^2-5}{35},\quad
 q_3=P_{-4,6}=\frac{16}{105} c \left(8 c^2-3\right).\notag
\end{gather}
\end{minipage}}
 
 We will show that they are  special cases of the associated ultraspherical polynomials.

For $c$ and $\nu$ two complex constants, if one lets $C_n^{(\nu)}(x;c)$ denote the {\it associated ultra-spherical polynomials} with initial conditions $C_{-1}^{(\nu)}(x;c)=0$, $C_0^{(\nu)}(x;c)=1$ then they are defined so as to satisfy the difference equation 
\begin{equation}
2x(n+\nu+c)C_{n}^{(\nu)}(x;c)=(n+c+1)C_{n+1}^{(\nu)}(x;c)+(2\nu+n+c-1)C_{n-1}^{(\nu)}(x;c).
\end{equation} 
Then setting $c=3/2$ and $\nu=-1/2$ this equation is the same recurrence relation as \eqnref{qrecursion}. 
Hence, we can conclude that polynomials $P_{-4,n}$ is a special case of associated ultraspherical polynomials.

For $\alpha,\beta,c\in\mathbb C$ constants, the {\it associated Jacobi polynomials} $P_n^{(\alpha,\beta)}(x;c)$ are defined so as to satisfy the recurrence relation
\begin{align}\label{assocJacobipolynomials}
2&(n+c+1)(n+c+\gamma)(2n+2c+\gamma-1)p_{n+1} \\
&=(2n+2c+\gamma)[(2n+2c+\gamma-1)(2n+2c+\gamma+1)x\notag\\
&\hskip 100pt +(\gamma-1)(\gamma-2\beta-1)]p_n\notag \\
&\quad -2(n+c+\gamma-\beta-1)(n+c+\beta)(2n+2c+\gamma+1)p_{n-1}\notag
\end{align}
where $n\in\mathbb N$, $\gamma=\alpha+\beta+1$, $P_{-1}^{(\alpha,\beta)}(x,c)=0$; $P_{0}^{(\alpha,\beta)}(x,c)=1$.

One can check that the associated ultraspherical polynomials are related to the associated Jacobi polynomials through the formula  
\begin{align*}
P_n^{(\nu-1/2,\nu-1/2)}(x;\beta)=\frac{\left(\nu+c+\frac{1}{2}\right)_n}{(2\nu+c)_n}C_n^{(\nu)}(x;\beta).
\end{align*}

%
 The weight measure for the associated ultraspherical polynomials is known (\cite{MR663314}) and is given by
 \begin{equation}
 \rho(\cos t)=\frac{(2\sin t)^{2\beta-1}[\Gamma(\lambda+\gamma)]^2}{2\pi \Gamma(2\beta+\gamma)\Gamma(\gamma+1)}|{_2F_1}(1-\beta,\gamma;\gamma+\beta;e^{2\imath t})|^{-2}
 \end{equation}
 for $0\leq t\leq \pi$ and where $_2F_1(a,b;c;z)$ is the hypergeometric function
 \begin{align*}
 _2F_1(a,b;c;z)=\sum_{n=0}^\infty\frac{(a)_n(b)_n}{(c)_n}z^n.
 \end{align*}

\begin{cor}\label{cor-orthog-non-clas}
Polynomials $P_{-4,n}$ are non-classical orthogonal polynomials.

\end{cor}

 Finally we should point out that Wimp (see  \cite{MR2191786}, equations 5.7.10-5.7.12  and formula (48) in \cite{MR915027}) using MACSYMA and ideas from \cite{MR1349110},  deduced that the associated Jacobi Polynomials satisfy the fourth order differential equation 
 \begin{equation}
 A_0(x)y^{(vi)}+A_1(x)y'''+A_2(x)y''+A_3(x)y'+A_4(x)y=0
 \end{equation} 
 where
 \begin{align*}
 A_0(x)&=(1-x^2)^2 \\
 A_1(x)&=10x(x^2-1) \\
 A_2(x)&=-(1-x)^2(2K+2C+\gamma^2-25)+2(1-x)(2K+2C+2\alpha \gamma) \\
 	&\qquad +2(\alpha+1)-26  \\
A_3(x)&=3(1-x)(2K+2C+\gamma^2-5)-6(K+C+\alpha \gamma+\beta-2) \\
A_4(x)&=n(n+2)(n+\gamma+2c)(n+\gamma+2c-2),
 \end{align*}
 where 
 $$
 K=(n+c)(n+\gamma+c),\qquad C=(c-1)(c+\alpha+\beta).
 $$
 
 In the setting of the polynomials $q_{n}(x)=P_{-4,2n+4}(x)=C^{(-1/2)}_{n}(x,3/2)$ for $n\geq -1$ , the fourth order differential equation we arrived at takes the form 
\begin{align} 
  &( x^2-1)^2   q_{n}^{(iv)}+10 x  (x^2-1 )  q_{n}''' -  (x^2  (2n^2+4n-23 )-2n^2-4n+11 )q_{n}''  \notag  \\
&\quad -3\cdot  x (2n^2+4n-3)q_{n}'  +n^2 (n+2)^2q_{n}  =0 \notag \\ \notag
\end{align}
For $n=2$ we get $\displaystyle{q_2(x)=\frac{32x^2-5}{35}}$, and plugging this into the equation above we have
\begin{align} 
  &-  (-7x^2 -5 )64  -3\cdot  x \cdot 13\cdot 64x  +4 \cdot 16(32x^2-5)  =0 \notag \\ \notag
\end{align}

In Ismail's notation we have 
$c=3/2$, $\nu=-1/2$, $\alpha=\beta=\gamma=-1$, 
 $$
 K=(n+3/2)(n+1/2),\quad C=-1/4 
 $$
 and 
\begin{align*}
 A_2(x)
	&=-(2 n^2+4 n-23) x^2- 52 x+2 n^2+4 n +3  \\
A_3(x)
&=-3(2n^2+4n-3)x \\
A_4(x)&=n^2(n+2)^2.
 \end{align*}
When $n=2$ we see that $A_2(x)$ is slightly off from our corresponding coefficient and
 $\displaystyle{q_2(x)=\frac{32x^2-5}{35}}$ does not satisfy the above equation. 
So there seems to be a typo in the formulation of $A_2(x)$ given in \cite{MR2191786} and the families of fourth order differential equations (in the variable $x$) given in the book by Ismail for the associated ultraspherical polynomials are different from the families of fourth order linear differential equations (in the variable $c$) derived in our work above, i.e. \eqnref{4thorder1} and \eqnref{4thorder2}.

\section{Orthogonality of $P_{-2,n}$}

Polynomials $P_{-2,n}$ satisfy the same recurrence relation as polynomials $P_{-4,n}$ but have different initial conditions.
Hence, we do know immediately whether they are associated ultraspherical polynomials. Here we will give an independent proof of the orthogonality of these polynomials.

  For $q_s(c)=P_{-2,2s}(c)$ we have
\begin{gather*}
q_{-1}=P_{-2,-2}=1,\quad q_{0}=P_{-2,0}=0,\quad q_1=P_{-2,2}=1/5,  \\ q_2=P_{-2,4}=8c/35,\quad
 q_3=P_{-2,6}=\frac{32c^2-7}{105}.
\end{gather*}

If $q_s(c)=P_{-2,2s}(c)$ we want polynomials with index $n$ giving the degree of the polynomial, then we will set
$$
\bar q_n:=q_{n+1},\quad n\geq -1
$$
while we ignore the ``first''  two polynomials; $q_{-1}=1$ and $q_0=0$.
Then
\vskip 5pt \noindent
\fbox{\begin{minipage}{5in}\begin{equation}\label{q-2}
\bar q_0=1/5,\quad \bar q_1=8c/35,\quad \bar q_2=\frac{32c^2-7}{105},\quad \bar q_3=\frac{8c(64c^2-29)}{1155}
\end{equation}\end{minipage}}
and \eqnref{qrecursion}, becomes
\begin{equation}
4(n+2)cq_{n+1}=(2n+7)q_{n+2}+ (2n+1)q_{n},
\end{equation}
or
\vskip 5pt
\noindent
\fbox{\begin{minipage}{5in}
\begin{equation}\label{qbarrecursion}
4(n+2)c\bar q_{n}=(2n+7)\bar q_{n+1}+ (2n+1)\bar q_{n-1},
\end{equation}\end{minipage}}
where the polynomials $\bar q_n$ are of degree $n$ in $c$.
\color{black}

We recall the following result by Favard which can be found in \cite{favard} or in Theorem 4.4 of \cite{MR0481884}:
\begin{thm} Let $\{\phi_n, n\geq0\}$ be a sequence of polynomials (monic) where $p_n$ is a polynomial of degree $n$, satisfying the following recursion
\begin{equation}\label{Favard2}
\phi_n=(x-\mu_n)\phi_{n-1}-\lambda_n\phi_{n-2},\qquad n=1,2,3,\dots
\end{equation}
for some complex numbers $\mu_n, \lambda_n$, $n=1,2,3,\dots$ and $\phi_{-1}=0$, $\phi_0=1$. Then $\{\phi_n, n\geq 0\}$ is a sequence of orthogonal polynomials with respect to a (unique) weight measure if and only if $\mu_n\in\mathbb R$ and $\lambda_{n+1}>0$ for all $n\ge1$.
\end{thm}

We have the following equivalent version of the theorem above.

\begin{thm}\label{favard}
 Let $\{p_n, n\geq 0\}$ be a sequence  of polynomials where $p_n$ is a polynomial of degree $n$, satisfying the following recursion
\begin{equation}\label{Favard1}
xp_n=a_{n+1}p_{n+1}+b_np_n+c_{n-1}p_{n-1},\qquad n=0,1,2,\dots
\end{equation}
for some complex numbers $a_{n+1}, b_n, c_{n-1}$, $n=0,1,2,\dots$and $p_{-1}=0$. Then $\{p_n, n\geq 0\}$ is a sequence of  orthonormal polynomials with respect to some (unique) weight measure if and only if $b_n\in\mathbb R$ and $c_{n}=\bar{a}_{n+1}\ne0$ for all $n\ge0$.
\end{thm}

\begin{proof} The equivalence of this statement with the original result of Favard is established in the following way: 
Let $k_n=||\phi_n||$  
and put $p_n=k_n^{-1}\phi_n$. Then from \eqref{Favard2} we get 
$$xp_n=k_n^{-1}k_{n+1}p_{n+1}+\mu_{n+1}p_n+k_n^{-1}\lambda_{n+1}k_{n-1}p_{n-1}.$$
By putting $a_{n+1}=k_n^{-1}k_{n+1}$, $b_n=\mu_{n+1}$ and $c_{n-1}=k_n^{-1}\lambda_{n+1}k_{n-1}$ we obtain \eqref{Favard1} with $b_n\in\mathbb R$ and $c_{n}=\bar{a}_{n+1}\ne0$ for all $n\ge0$, because 
$$c_n=(xp_{n+1},p_n)=(p_{n+1},xp_n)=\bar a_{n+1}.$$

Conversely, let $d_n$ be the leading coefficient of $p_n$ and put $\phi_n=d_n^{-1}p_n$.  Then from \eqref{Favard1} we get
 $$x\phi_{n-1}=d_{n-1}^{-1}a_{n}d_n\phi_n+b_{n-1}\phi_{n-1}+d_{n-1}^{-1}c_{n-2}d_{n-2}\phi_{n-2}.$$
 Since the $\phi_n$'s are monic we have $d_{n-1}^{-1}a_nd_n=1$. Therefore 
we obtain \eqref{Favard2} with $\mu_n=b_{n-1}\in\mathbb R$ and  $\lambda_{n+1}=d_n^{-1}c_{n-1}d_{n-1}=a_{n}c_{n-1}=a_{n}\bar a_{n}>0$ for all $n\ge1$. 
\end{proof}

We apply Theorem~\ref{favard} to our setting.
 If we set  $\displaystyle{a_{n+1}=\frac{2n+7}{4(n+2)}}$, $b_n=0$, $\displaystyle{c_{n-1}=\frac{2n+1}{4(n+2)}}$, then the hypothesis of Favard's Theorem is not satisfied.  But with a modification to our recursion formula we show that the $\bar q_n$ are orthogonal. 

\begin{thm}\label{thm-orth}
 The polynomials $P_{-2,n}(c)$ are orthogonal with respect to some weight function.
\end{thm}

\begin{proof}
 It is sufficient to check that there exist a family of orthonormal polynomials $f_n$ and constants $\lambda_n$ such that $ q_n=\lambda_n f_n$ for all $n$


We have the following recursion for $f_n$:
$$
4(n+2)cf_n=(2n+7)\lambda_{n+1}f_{n+1}+ (2n+1)\lambda_{n-1}f_{n-1}$$
or
$$
cf_n=\frac{(2n+7)\lambda_{n+1}}{4(n+2)\lambda_n}f_{n+1}+ \frac{(2n+1)\lambda_{n-1}}{4(n+2)\lambda_n}f_{n-1},$$
for $n\geq 1$.
\color{black}

%
Set
$$
A_{n+1}=\frac{(2n+7)\lambda_{n+1}}{4(n+2)\lambda_n},\quad
C_{n-1}=\frac{(2n+1)\lambda_{n-1}}{4(n+2)\lambda_n}.
$$
Then
$A_n=C_{n-1}$ if and only if
 $\displaystyle{\frac{(2n+5)\lambda_{n}}{4(n+1)\lambda_{n-1}}=\frac{(2n+1)\lambda_{n-1}}{4(n+2)\lambda_n}}$
 or
$$\lambda_n^2=\frac{(n+1)(2n+1)}{(n+2)(2n+5)}\lambda_{n-1}^2.$$

\color{black}
Taking $\lambda_0=1$ we can find a family of constants $\lambda_n$ satisfying this relation. Hence, by Theorem~\ref{favard}
the polynomials $f_n$ form an orthonormal family with respect to some measure, and therefore the polynomials are orthogonal.

\color{black}

\end{proof}

One can also show that $P_{-2,n}$ are non-classical orthogonal polynomials. We will give a simple proof of this fact in  Appendix. The same argument works in the case of $P_{-4,n}$.

\section{Appendix}


Given a sequence of orthogonal polynomials $\{q_n\}_{n\ge0}$ we consider the complex linear space
of all differential operators with complex coefficients on the real line that have the polynomials $q_n$ as their eigenfunctions. Thus
$$\mathcal{D}(w)=\{D:Dq_n=\gamma_n(D) q_n, \; \gamma_n(D)\in{\bf C}\;\;\text{for all}\; n\ge0\}.$$

Some properties of $\mathcal{D}(w)$, see \cite{MR2329130}:

(1) The definition of $\mathcal{D}(w)$ depends only on the weight function $w=w(x)$ and not on the orthogonal sequence $\{q_n\}_{n\ge0}$.

(2) If $D\in\mathcal{D}(w)$, then
$$D=\sum_{i=0}^s f_i(x)\left(\frac{d}{dx}\right)^i,$$
where $f_i(x)$ is a polynomial and $\deg f_i\le i$.

To ease the notation if $\nu\in{\bf C}$ let
$$[\nu]_i=\nu(\nu-1)\cdots(\nu-i+1),\qquad [\nu]_0=1.$$

(3) If $\{q_n\}_{n\ge0}$ is a sequence of orthogonal polynomials and $D\in\mathcal{D}(w)$ is of the form $D=\sum_{i=0}^sf_i(x)\left(\frac{d}{dx}\right)^i$, with
$f_i(x)=\sum_{j=0}^i f_j^i(D)x^j,$
then
$$\gamma_n(D)=\sum_{i=0}^s[n]_if_i^i(D).$$

Let us consider the Weyl algebra ${\bf D}=\{D=\sum_{i=0}^s f_i(x)\left(\frac{d}{dx}\right)^i:f_i\in{\bf C}[x]\}$,
and the subalgebra
$$\mathcal{D}=\{D=\sum_{i=0}^s f_i(x)\left(\frac{d}{dx}\right)^i\in{\bf D}: \deg(f_i)\le i\}.$$
(4) If $D\in\mathcal{D}$ satisfies the symmetry condition $(Dp,q)=(p,Dq)$ for all $p,q\in{\bf C}[x]$, then $D\in\mathcal{D}(w)$.

(5) For any $D\in\mathcal{D}(w)$ there is a unique $D^*\in\mathcal{D}(w)$ such that $(Dp,q)=(p,D^*q)$ for all $p,q\in{\bf C}[x]$. We refer to $D^*$ as the adjoint of $D$. The map $D\mapsto D^*$ is a *-operation in the algebra $\mathcal{D}(w)$, and the orders of $D$ and $D^*$ coincide.

(6) The set $\mathcal{S}(w)=\{D\in \mathcal{D}(w): D=D^*\}$ of all symmetric differential operators is a real form of $\mathcal{D}(w)$:
$$\mathcal{D}(w)=\mathcal{S}(w)\oplus i\mathcal{S}(w).$$

(7) $D\in\mathcal{D}(w)$ is symmetric if and only if $\gamma_n(D)$ is real for all $n\ge0$.

In the literature a weight $w$ is called classical if there exists a second order symmetric differential operator $D$ such that $Dq_n=\gamma_n q_n$ for all $n\ge0$.

\

\begin{thm}\label{nonclassical}
The sequences of orthogonal polynomials \eqnref{q-2} $\{\bar q_n\}_{n\ge0}$ and \eqnref{q-4} $\{ q_n\}_{n\ge0}$ are not classical.
\end{thm}
\begin{proof}  In the polynomials above we replace $c$ by $x$ as it is more natural to use this later letter as our variable.

 From (6) above, to prove the theorem is equivalent to prove that in $\mathcal{D}(w)$ there is no differential operator of order 2. Suppose $D\in\mathcal{D}(w)$ is of order 2. Then we can assume that $D$ is of the form
\begin{equation}\label{2ndorderode}
D=(ax^2+bx+c)\left(\frac{d}{dx}\right)^2+(ex+f)\frac{d}{dx},
\end{equation}
for some $a,b,c,e,f\in{\bf C}$.
For \eqnref{q-2} the first five terms of the orthogonal sequence $\{\bar q_n\}_{n\ge0}$ are:
$$\bar q_0=\frac{1}{5},\quad \bar q_1=\frac{8x}{35},\quad \bar q_2=\frac{32x^2-7}{105},\quad \bar q_3=\frac{8x(64x^2-29)}{1155},$$
$$\bar q_4=\frac{(160\times64) x^4-(32\times222)x^2+(77\times7)}{13\times1155}.$$
From the definition of $\mathcal{D}(w)$ and (3) we know that
\begin{equation}
\label{2ndorderclassical}
(ax^2+bx+c)\left(\frac{d}{dx}\right)^2\bar q_n+(ex+f)\frac{d}{dx}\bar q_n=(an(n-1)+bn+c+en+f)\bar q_n,
\end{equation}
for all $n\ge0$.

If we set $n=0$ we get $c+f=0$.  If we put $n=1$ in the above equation we get
$$
(ex+f)\frac{8}{35}=(b+c+e+f)\frac{8x}{35}
$$
which implies $f=0$ and so $c=0$ (as $c+f=0$) and $e=b+e$. Hence $b=0$.
Then for $n=2$ we obtain
\begin{align*}
(2a+2e)\frac{32x^2-7}{105}&=ax^2\left(\frac{d}{dx}\right)^2\left(\frac{32x^2-7}{105}\right)+ex\frac{d}{dx}\left(\frac{32x^2-7}{105}\right) \\
&=ax^2 \frac{64}{105}+ex\frac{64x}{105}
\end{align*}
and hence $ a+e=0$. Using all this information, the above equation \eqnref{2ndorderclassical} for $n=3$ is equivalent to
\begin{align*}
(6a+3e)  \frac{8x(64x^2-29)}{1155}&=ax^2\left(\frac{d}{dx}\right)^2\left(\frac{8x(64x^2-29)}{1155}\right)+ex\frac{d}{dx}\left( \frac{8x(64x^2-29)}{1155}\right)  \\
&= \frac{8\cdot 64\cdot 6\cdot ax^3}{1155}+ \frac{8\cdot 64\cdot 3\cdot ex^3}{1155}-\frac{8\cdot 29ex}{1155}
\end{align*}
which reduces to
$$
2a+e=0.
$$
Therefore  $a=e=0$ which implies that $\mathcal{D}=0$. In other words we have proved that the only differential operators in $\mathcal{D}(w)$ of degree less or equal to two are the constants.

For \eqnref{q-4} \begin{gather}
 q_0 =1,  \quad q_1 =4x/5,\quad
 q_2=\frac{32x^2-5}{35},\quad
 q_3=\frac{16}{105} x \left(8 x^2-3\right).\notag
\end{gather}
 A similar analysis yields also $\mathcal{D}=0$.
%
%
%
\end{proof}


\section{Acknowledgment}
   Second author was
supported in part by the CNPq grant (301743/2007-0), by the
Fapesp grant (2010/50347-9) and by the MathAmSud grant (611/2012).

%
%
\def\cprime{$'$} \def\cprime{$'$}

\end{document}